\documentclass[]{article}
\usepackage{amssymb,amsmath,amsthm,mathdesign, color}
\usepackage{pdfsync}

\newcommand{\R}{\mathbf{R}}

\newcommand {\E}{\mathrm{E}}

\renewcommand{\d}{\text{\rm d}}

\newcommand{\sG}{\mathcal{G}}

\newcommand{\sE}{\mathcal{E}}

\newcommand{\ve}{\epsilon}

\newtheorem{stat}{Statement}[section]
\newtheorem{proposition}[stat]{Proposition}
\newtheorem{corollary}[stat]{Corollary}
\newtheorem{theorem}[stat]{Theorem}
\newtheorem{lemma}[stat]{Lemma}

\theoremstyle{definition}

\numberwithin{equation}{section}

\begin{document}

\title{On the behaviour of stochastic heat equations on bounded domains.}

\author{Mohammud Foondun\thanks{Research supported in part by EPSRC}\\
\small Department of Mathematics\\[-0.8ex]
\small Loughborough University\\
\small Leicestershire, LE11 3TU, UK\\
\small \texttt{m.i.foondun@lboro.ac.uk}\\
\and
Eulalia Nualart\thanks{Research supported in part by the European Union programme FP7-PEOPLE-2012-CIG under grant agreement 333938}\\
\small Department of Economics and Business\\[-0.8ex]
\small Universitat Pompeu Fabra and Barcelona Graduate School of Economics\\[-0.8ex]
\small Ram\'on Trias Fargas 25-27, 08005
Barcelona, Spain\\
\small \texttt{eulalia@nualart.es}
}

\date{November, 2014}
\maketitle
\begin{abstract}
Consider the following equation $$\partial_t u_t(x)=\frac{1}{2}\partial _{xx}u_t(x)+\lambda \sigma(u_t(x))\dot{W}(t,\,x)$$ on an interval.	Under Dirichlet boundary condition, we show that in the long run, the second moment of the solution grows exponentially fast if $\lambda$ is large enough. But if $\lambda$ is small, then the second moment eventually decays exponentially.  If we replace the Dirichlet boundary condition by the Neumann one, then the second moment grows exponentially fast no matter what $\lambda$ is.  We also provide various extensions.
		
		
\end{abstract}
\newpage

\section{Introduction and main results.}

Fix $L>0$ and consider the following stochastic heat equation on the interval $(0,L)$ with Dirichlet boundary condition:
\begin{equation}\label{eq:dirichlet}
\left|\begin{aligned}
&\partial_t u_t(x)=\frac{1}{2}\partial _{xx}u_t(x)+\lambda \sigma(u_t(x))\dot{W}(t,\,x)\;\; \text{for}\;\; 0<x<L\;\;\text{and}\; \;t>0\\
&u_t(0)=u_t(L)=0 \quad \text{for}\quad t>0,
\end{aligned}\right.
\end{equation}
where the initial condition $u_0:[0,L] \rightarrow \R_+$ is a non-negative bounded function with positive support inside the interval $[0,\,L]$. Here $\dot{W}$ denotes space-time white noise and $\sigma:\R\rightarrow \R$ is a globally Lipschitz function satisfying $l_\sigma|x|\leq  |\sigma(x)| \leq L_\sigma|x|$ where $l_\sigma$ and $L_\sigma$ are positive constants. 

The main aim of this paper is to investigate the long time behaviour of the solution to the above equation with respect to $\lambda$, a positive parameter which we call {\it the level of the noise}.  

As usual, we follow Walsh \cite{Walsh} to say that $u_t$ is a mild solution to \eqref{dirichlet} if \begin{equation}\label{mild:dirichlet}
u_t(x)=
(\sG_Du)_t(x)+ \lambda \int_0^L\int_0^t p_D(t-s,\,x,\,y)\sigma(u_s(y))W(\d s\,,\d y),
\end{equation}
where $p_D(t,x,y)$ denotes the Dirichlet heat kernel, and
\begin{equation*}
(\sG_D u)_t(x):=\int_0^L u_0(y)p_D(t,\,x,\,y)\,\d y.
\end{equation*}
See Walsh \cite{Walsh} or \cite{minicourse} for proofs of existence and uniqueness. Our first theorem reads as follows. 
\begin{theorem}\label{dirichlet}
There exists $\lambda_0>0$ such that for all $\lambda<\lambda_0$ and $x \in (0,L)$, 
\begin{equation*}
-\infty<\limsup_{t\rightarrow \infty}\frac{1}{t}\log \E|u_t(x)|^2<0.
\end{equation*}
On the other hand, for all $\ve>0$, there exists $\lambda_1>0$ such that for all $\lambda>\lambda_1$, and $ x\in[\ve, L-\ve]$,
\begin{equation*}
0<\liminf_{t\rightarrow \infty}\frac{1}{t}\log\E|u_t(x)|^2<\infty.
\end{equation*}
\end{theorem}

In \cite{Khoshnevisan:2013ab}, the authors looked at the behaviour of the solution to \eqref{dirichlet} as $\lambda$ gets large.  This project grew out of trying to understand the behaviour of the solution when one keeps $\lambda$ fixed but take $t$ to be large.  So our results complement those in \cite{Khoshnevisan:2013ab}. We should mention that we could not get precise estimate for $\lambda_0$ and $\lambda_1$, although our proof shows that they  depend on the first eigenvalue of the Dirichlet Laplacian; a fact which is not surprising.  However, we have the following result.
\begin{corollary}\label{zero}
We have $\lambda_0\leq \lambda_1$. Moreover,
 there exist $\bar{\lambda} \in [\lambda_0, \lambda_1]$ such that
\begin{equation*}
\limsup_{t\rightarrow \infty}\frac{1}{t}\log\E|u_t(x)|^2=0,
\end{equation*}
and $\tilde{\lambda} \in [\lambda_0, \lambda_1]$ such that
\begin{equation*}
\liminf_{t\rightarrow \infty}\frac{1}{t}\log\E|u_t(x)|^2=0.
\end{equation*}
Moreover, $\bar{\lambda} \leq \tilde{\lambda}$.
\end{corollary}

We have another corollary of the above theorem, but first we need a definition which we borrow from \cite{Khoshnevisan:2013ab}. We define the energy of the solution $u$ as 
\begin{equation}\label{energy}
\sE_t(\lambda):=\sqrt{\E \Vert u_t \Vert^2_{L^2(0,L)}}.
\end{equation}
%Set 
%$$
%\mathcal{L}_2(\lambda)=\limsup_{t\rightarrow \infty}\frac{1}{t}\log \sE_t(\lambda).
%$$
Our corollary is the following.
\begin{corollary}\label{cor:dirichlet}
Let $\lambda_0$ and $\lambda_1$ as in Theorem \ref{dirichlet}. Then 
\begin{equation*}
-\infty<\limsup_{t\rightarrow \infty}\frac{1}{t}\log \sE_t(\lambda)<0 \quad\text{for\,\,all}\quad \lambda< \lambda_0,
\end{equation*}
and 
\begin{equation*}
0<\liminf_{t\rightarrow \infty}\frac{1}{t}\log \sE_t(\lambda)<\infty \quad\text{for\,\,all}\quad \lambda> \lambda_1.
\end{equation*}
\end{corollary}

As an extension, we consider the following equation with linear drift. 
\begin{equation}\label{eq:dirichlet-drift}
\left|\begin{aligned}
&\partial_t u_t(x)=\frac{1}{2}\partial _{xx}u_t(x)+\mu u_t(x)+\lambda \sigma(u_t(x))\dot{W}(t,\,x)\quad\text{for}\quad0<x<L,\quad t>0\\
&u_t(0)=u_t(L)=0 \quad \text{for}\quad t>0,
\end{aligned}\right.
\end{equation}
where $\mu$ is a real number and all the other conditions are the same as before. We then have the following result.

\begin{theorem}\label{dirichlet-drift}
There exists $\lambda_0(\mu)>0$ such that for all $\lambda<\lambda_0(\mu)$ and $x \in (0,L)$, 
\begin{equation*}
-\infty<\limsup_{t\rightarrow \infty}\frac{1}{t}\log \E|u_t(x)|^2<0.
\end{equation*}
On the other hand, for all $\ve>0$, there exists $\lambda_1(\mu)>0$ such that for all $\lambda>\lambda_1(\mu)$, and $ x\in[\ve, L-\ve]$,
\begin{equation*}
0<\liminf_{t\rightarrow \infty}\frac{1}{t}\log\E|u_t(x)|^2<\infty.
\end{equation*}
Moreover when $\mu>0$, both $\lambda_0(\mu)$ and $\lambda_1(\mu)$ are decreasing with $\mu$. Otherwise, they are both increasing.
\end{theorem}

The proof of Theorem \ref{dirichlet} shows that there is some kind of fight between the "dissipative" effect of the Dirichlet Laplacian and the noise term which tend to push the second moment exponentially high. When $\lambda$ is small enough, the noise term is not enough to induce exponential growth of the second moment.  If we replace the Dirichlet boundary condition by Neumann boundary condition, we have a different behaviour. Consider the following equation,
\begin{equation}\label{eq:neumann}
\left|\begin{aligned}
&\partial_t u_t(x)=\frac{1}{2}\partial _{xx}u_t(x)+\lambda \sigma(u_t(x))\dot{W}(t,\,x)\quad\text{for}\quad0<x<L\quad\text{and}\quad t>0\\
&\partial_xu_t(0)=\partial_xu_t(L)=0 \quad \text{for}\quad t>0.
\end{aligned}
\right.
\end{equation}
All other conditions are the same as those for \eqref{dirichlet}. We follow Walsh \cite{Walsh} again to define the {\it mild solution} to \eqref{eq:neumann} by the following integral equation 
\begin{equation}
u_t(x)=
(\sG_Nu)_t(x)+ \lambda \int_0^L\int_0^t p_N(t-s,\,x,\,y)\sigma(u_s(y))W(\d s\,,\d y),
\end{equation}
where
\begin{equation*}
(\sG_N u)_t(x):=\int_0^L u_0(y)p_N(t,\,x,\,y)\,\d y,
\end{equation*}
and $p_N(t,\,x,\,y)$ denotes the Neumann heat kernel. Our main result concerning the above equation is given be the following. 
\begin{theorem}\label{neumann}Let $u_t$ denote the unique solution to \eqref{neumann}, then for all $x\in (0,\,L)$,
\begin{equation*}
\E|u_t(x)|^2\geq c_1e^{c_2\lambda^4t}\quad \text{for\,\,all}\quad t>0.
\end{equation*}
In particular, we have 
\begin{equation*}
0<\liminf_{t\rightarrow \infty}\frac{1}{t}\log \E|u_t(x)|^2<\infty \quad\text{for\,all}\quad \lambda>0.
\end{equation*}
\end{theorem}

The above says that no matter how small the noise term is, the second moment will grow exponentially fast. We have an immediate corollary.
\begin{corollary}\label{cor:neumann}
The energy $\sE_t(\lambda)$ of the solution to \eqref{eq:neumann}, satisfies 
\begin{equation*}
0<\liminf_{t\rightarrow \infty}\frac{1}{t}\log \sE_t(\lambda)<\infty\quad{for\,\,all}\quad \lambda>0.
\end{equation*}
\end{corollary} 

For the corresponding equation on the whole line, it is known that the second moment grows exponentially no matter what $\lambda$ is. See for instance \cite{CoKh}. Therefore our results show that the Dirichlet heat equation exhibit a different phenomena from the other two types of equations. This is a new result. As far as we know, no such result has been proved before. 

Intuitively, our results show that if the amount of noise is small, then the heat lost in a system modelled by the Dirichlet equation is not enough to increase the energy in the long run.  On the other hand, if the amount of noise is large enough, then the energy will increase.

The parameter can also be viewed as the inverse of the temperature and the solution $u$ can be regarded as the partition function of a continuous time and space random polymer.  We refer the reader to \cite{bezerra} and references therein.

The above theorems are about second moments. Similar results hold for higher moments as well. We defer the statement and proof to Section 6.  For the rest of this introduction, we focus on higher dimensional analogues of the results described above.  Fix $R>0$ and consider the stochastic heat equation on the $d$-dimensional ball $\mathcal{B}(0,R)$ with $ d\geq 1$,
\begin{equation}\label{eq:dir}
\left|\begin{aligned}
&\partial_t u_t(x)=\frac12 \Delta u_t(x)+\lambda \sigma(u_t( x)) \dot{F}(t,\,x)\quad x \in \mathcal{B}(0,R)\quad t>0\\
&u_t(x)=0 \quad x \in \mathcal{B}(0,R)^c, 
\end{aligned}
\right.
\end{equation}
where the initial condition $u_0$ is assumed to be a non-negative bounded function
 with positive support in the clousure of the ball $\mathcal{B}(0,R)$. The Gaussian noise $\dot{F}(t,x)$ is white in time and coloured in space,
that is,  
$$
\E \left( \dot{F}(t,x) \dot{F}(s,y)\right)=\delta_0(t-s) f(x-y),
$$
where $f$ is a locally integrable positive continuous function on $\mathcal{B}(0,R) \setminus \{0\}$, and maybe unbounded at the origin.
We assume that $\sigma$ is a globally Lipschitz function satisfying $l_\sigma|x|\leq  \sigma(x) \leq L_\sigma|x|$ where $l_\sigma$ and $L_\sigma$ are positive constants.  Observe that we are not taking an absolute value to $\sigma(x)$ because we do not know if the solution to this equation is non-negative a.s.

The mild solution satisfies the following integral equation \begin{equation}\label{mild:coloured}
u_t(x)=
(\sG_Du)_t(x)+ \lambda \int_{\mathcal{B}(0,R)}\int_0^t p_D(t-s,\,x,\,y)\sigma(u_s(y))F(\d s,\,\d y),
\end{equation}
where $p_D(t,x,y)$ denotes the Dirichlet heat kernel on the ball $\mathcal{B}(0,R)$, and the integral is understood
in the sense of Walsh \cite{Walsh}; see also Dalang \cite{Dalang}.  Well-posedness of the solution holds under the following additional integrability condition.  For any fixed $\epsilon>0$, we have 
\begin{equation*}
\int_{\vert x \vert \leq \epsilon} f(x) \log \frac{1}{\vert x \vert} dx < \infty, \quad \text{ if } d=2,
\end{equation*}
or
\begin{equation*} 
\int_{\vert x \vert \leq \epsilon} f(x) \frac{1}{\vert x \vert^{d-2}} dx < \infty, \quad \text{ if } d\geq 3.
\end{equation*}
No such integrability condition is needed when $d=1$; see \cite{Dalang}.  We defer the proof of well-posedness to the appendix.  Here are our main results concerning the above equation.
\begin{theorem}\label{coloured}
Suppose that $u_t$ is the unique solution to \eqref{eq:dir}. Then there exists $\lambda_0>0$ such that for all $\lambda<\lambda_0$ and $x \in \mathcal{B}(0,R)$, 
\begin{equation*}
-\infty<\limsup_{t\rightarrow \infty}\frac{1}{t}\log \E|u_t(x)|^2<0.
\end{equation*}
On the other hand, for all $\ve>0$, there exists $\lambda_1>0$ such that for all $\lambda>\lambda_1$, and $x\in \mathcal{B}(0,R-\ve)$,
\begin{equation*}
0<\liminf_{t\rightarrow \infty}\frac{1}{t}\log \E|u_t(x)|^2<\infty.
\end{equation*}
\end{theorem}
We now have the following corollary. 
\begin{corollary}\label{cor:coloured}
Let $\lambda_0$ and $\lambda_1$ as in Theorem \ref{coloured}. Then 
\begin{equation*}
-\infty<\limsup_{t\rightarrow \infty}\frac{1}{t}\log \sE_t(\lambda)<0 \quad\text{for\,\,all}\quad \lambda< \lambda_0,
\end{equation*}
and 
\begin{equation*}
0<\liminf_{t\rightarrow \infty}\frac{1}{t}\log \sE_t(\lambda)<\infty \quad\text{for\,\,all}\quad \lambda>\lambda_1.
\end{equation*}
\end{corollary}
Our results explore the effect of the level of noise on the long time behaviour of the second moment of the solution.  This current paper is the first to explore such questions. Using the terminology introduced in \cite{FK}, our results say that under some conditions, the solution to the equations studied above is "weakly intermittent".  To the best of our knowledge,  this is the first paper about intermittency properties of the equations on bounded domains.  Using Feymann-Kac formula, almost sure long time behaviour has been investigated for related equations, see for instance \cite{Carmona} and references therein.  If one replaces the white noise by a one parameter Gaussian noise, similar results can be found in \cite{AK}, \cite{BinMin} and \cite{Bin}. A study of invariant measures for similar equations is done in \cite{cerrai}.

Before we give a plan of the article, we mention a few words about the method. This paper grew out of understanding \cite{Khoshnevisan:2013ab}. But the techniques used here are completely different. Here, we use pointwise bounds on the heat kernel to obtain the required estimates.  We heavily use the fact for small times, the Dirichlet heat kernel behaves likes the Gaussian ones, while for large times, it decays exponentially. For the Neumann problem, we also compared the corresponding heat kernel with the Gaussian one. As such, the idea of comparing the Dirichlet heat kernel with the Gaussian one is not new but here we employ it to find long time behaviour of the solution, while in \cite{foonjose}, this idea has been used to find asymptotic behaviour of the second moment with respect to $\lambda$. We should also mention that, dealing with the coloured case equation is not straightforward. A new method had to be employed. See \cite{FK2} for a comparison. 

We now give a plan of the article. Section 2 contains all the required estimates to prove then main results. The main result concerning the Dirichlet problem is proved in Section 3.  Section 4 contains the analogous results for the Neumann problem.  In Section 5, we give the proofs for the equation in higher dimension. An extension to higher moments is proved in Section 6.  The Appendix contains an existence and uniqueness result.

\section{Some estimates.}

Our proofs will require some estimates involving the heat kernels.  We begin with the Dirichlet heat kernel. It is well known that
$$
p_D(t,x,y)=\sum_{n=1}^{\infty} e^{-\mu_n t} e_n(x) e_n(y)\quad\text{for}\quad x,\,y \in \Sigma,
$$
where $\Sigma$ is a bounded subset of $\R^d$. $\{e_k\}_{k \in \mathbb{N}}$ is a complete orthonormal system of $L^2(\Sigma)$ that diagonalizes the Dirichlet Laplacian in $L^2(\Sigma)$, that is,
\begin{equation}\label{eq:e}
\left|\begin{aligned}
&\Delta e_k(x)=-\mu_k e_k(x) \quad x \in \Sigma,\\
& e_k(x)=0 \quad x \in \Sigma^c,
\end{aligned}
\right.
\end{equation}
and $\{ \mu_k\}_{k \in \mathbb{N}}$ is an increasing sequence of positive  numbers. We have the following bounds which will be very useful. The first one is the following trivial bound
\begin{equation} \label{p1}
p_D(t,\,x,\,y)\leq p(t,\,x,\,y)\quad\text{for all}\quad t>0,
\end{equation}
where $p(t,\,x,\,y)$ is the Gaussian heat kernel. While this bound is very useful for small time $t$, it is useless for large time. In the latter case, we will use the following 
\begin{equation} \label{p2}
p_D(t,\,x,\,y)\leq c_1e^{-\mu_1 t}\quad\text{for}\quad t>t_0, 
\end{equation}
where $t_0$ is some positive number.  Provided that $x,\,y$ are within a positive distance from the boundary of the domain, the above inequality can be reversed as follows
\begin{equation*}
p_D(t,\,x,\,y)\geq c_2e^{-\mu_1 t}\quad\text{for}\quad t>t_0.
\end{equation*}
The use of the above bounds is key to our method. These inequalities hold in much more general settings, see for instance \cite{zhang}.  It should be noted that in the special case that $\Sigma:=(0,\,L)$, we have $\phi_n(x):=\left(\frac{2}{L}\right)^{1/2}\sin\left(\frac{n\pi x}{L} \right)$ and $\mu_n=\left(\frac{n\pi}{L}\right)^2$. 

We will need the following upper bounds.
\begin{lemma}\label{upperbound-1}
There exists a constant $c_1$ depending on $L$ such that for all $\beta \in (0,\mu_1)$ and $x \in (0,L)$, we have
\begin{equation*}
\int_0^\infty e^{\beta t}p_D(t,\,x,\,x)\,\d t\leq c_1[\frac{1}{\sqrt{\beta}}+\frac{1}{\mu_1-\beta}].
\end{equation*}
\end{lemma}

\begin{proof}
We use the bounds (\ref{p1}) and (\ref{p2}). 
We therefore write 
\begin{equation*}
\begin{aligned}
\int_0^\infty e^{\beta t}p_D(t,\,x,\,x)\,\d t
&=\int_0^{t_0}e^{\beta t}p_D(t,\,x,\,x)\,\d t+\int_{t_0}^\infty e^{\beta t}p_D(t,\,x,\,x)\,\d t\\
&=I_1+I_2.
\end{aligned}
\end{equation*}
We estimate $I_1$ first. Using (\ref{p1}),
\begin{equation*}
\begin{aligned}
I_1&\leq \int_0^{t_0} e^{\beta t}p(t,\,x,\, x)\,\d t\leq c_2\int_0^{t_0} \frac{e^{\beta t}}{\sqrt{t}}\,\d t.
\end{aligned}
\end{equation*}
As for $I_2$, by (\ref{p2}), we have
\begin{equation*}
\begin{aligned}
I_2 &\leq \int_{t_0}^\infty e^{\beta t}p_D(t,\,x,\,x)\,\d t \leq c_3\int_{t_0}^\infty e^{\beta t}e^{-\mu_1 t}\,\d t.
\end{aligned}
\end{equation*}
Combining these estimates yield the result.
\end{proof}

\begin{lemma}\label{upperbound-2}
There exists a constant $c_1$ depending on $L$ such that for all $\beta \in (0,\mu_1)$ and $x \in (0,L)$, we have 
\begin{equation*}
\sup_{t>0}\int_0^L e^{\beta t} p_D(t,\,x,\,y)\,d y\leq c_1.
\end{equation*}
\end{lemma}
\begin{proof}
As in the previous proof, we will deal with large and small times separately.  For $0<t\leq t_0$, we use $p_D(t,\,x,\,y)\leq p(t,\,x,\,y)$ to obtain
\begin{equation*}
\begin{aligned}
\int_0^L e^{\beta t} &p_D(t,\,x,\,y)\,d y\leq e^{\beta t}\int_\R p(t,\,x,\,y)\,d y\leq e^{\beta t_0}.
\end{aligned}
\end{equation*}
For $t>t_0$, we use $p_D(t,\,x,\,y)\leq c_2e^{-\mu_1t}$ and write 
\begin{equation*}
\begin{aligned}
\int_0^L e^{\beta t} &p_D(t,\,x,\,y)\,d y\leq c_2e^{-(\mu_1-\beta) t}\int_0^L \,d y\leq c_3.
\end{aligned}
\end{equation*}
Since $\beta<\mu_1$, the above two inequalities proves the result.
\end{proof}

The above two results will be used to establish upper bounds on quantities involving the second moment of the solution to \eqref{dirichlet}.  To establish analogous lower bounds, the following will be useful.

\begin{lemma}\label{lowerbound}
For all $\ve>0$, then there exists $t_0>0$ and $c_1>0$ depending on $\epsilon$ and $L$ such that for all $\beta>0$ and $x,\,y\in [\ve, L-\ve]$, we have 
\begin{equation*}
\int_0^\infty e^{-\beta t} p^2_D(t,\,x,\,y)\, \d t \geq \frac{c_1 e^{-(\beta+2\mu_1)t_0}}{\beta+2\mu_1}.
\end{equation*}
\end{lemma}

\begin{proof}
Since $x,\,y\in [\ve,L-\ve]$, there exists $t_0>0$ and $c_3>0$ depending on $\epsilon$ and $L$ such that
\begin{equation*}
p_D(t,\,x,\,y)\geq c_2e^{-\mu_1t}\quad \text{for\,\,all}\quad t>t_0.
\end{equation*}
Thus, we get the following
\begin{equation*}
\begin{aligned}
\int_0^\infty e^{-\beta t}p^2_D(t,\,x,\,y)\, \d t&\geq \int_{t_0}^\infty e^{-\beta t}p^2_D(t,\,x,\,y)\, \d t\\
&\geq \frac{c_3e^{-(\beta+2\mu_1) t_0}}{\beta+2\mu_1}.
\end{aligned}
\end{equation*}
\end{proof}

The above result is quite crude. But it will be enough for our purpose. To study \eqref{neumann}, we will need estimates involving the Neumann heat kernel. We have  
\begin{equation*}
p_N(t,\,x,\,y)=\frac{1}{L}+\sum_{n=1}^\infty e^{-\mu_n t}\cos \frac{n\pi x}{L}\cos \frac{n\pi y}{L},
\end{equation*}
and 
\begin{equation*}
p_N(t,\,x,\,y)\geq p(t,\,x,\,y)\quad \text{for\,\,all}\quad x,\,y\in [0,\,L],
\end{equation*}
where $p(t,\,x,\,y)$ denotes the heat kernel on the real line. The preceding inequality can be deduced from the fact that the Neumann heat kernel represents  the probability density function of reflected Brownian motion. We thus have 

\begin{lemma}\label{neu:lower}
There exists a constant $c_1$ such that for all $x\in [0,\,L]$
\begin{equation*}
p_N(t,\,x,\,y)\geq c_1\quad{for\,\,all}\quad t>0.
\end{equation*}
\end{lemma}

\begin{proof}
There exists a $T>0$ such that for $t>T$, we have 
\begin{equation*}
p_N(t,\,x,\,y)\geq \frac{1}{2L}\quad\text{for\,\,all}\quad t\geq T.
\end{equation*}
For $t<T$, we can use the fact that
\begin{equation*}
p_N(t,\,x,\,y)\geq p(t,\,x,\,y),
\end{equation*}
and since $x,\,y\,\in [0,\,L]$, we have $p_N(t,\,x,\,y)\geq c_1$. Combining the above estimates, we get our result.
\end{proof}
So far the estimates we have are basically one dimensional and will be helpful for studying \eqref{eq:dirichlet} and \eqref{eq:neumann}.  To study \eqref{eq:dir}, we will need some more estimates. Let us first observe that there exists a constant $c_1$ depending on $R$ such that for all $\beta \in (0,\mu_1)$ and $x \in \mathcal{B}(0,R)$, we have 
\begin{equation}\label{upperbound-4}
\sup_{t>0}\int_{\mathcal{B}(0,R)} e^{\beta t} p_D(t,\,x,\,y)\,d y\leq c_1.
\end{equation}
The proof is exactly the same as Lemma \ref{upperbound-2}.  Our next result is the following.
\begin{lemma}\label{upperbound-3}
There exists a constant $c_1$ depending on $R$ such that for all $\beta \in (0,2\mu_1)$ and all $x,\,y_1,\,y_2\in \mathcal{B}(0,R)$, we have
\begin{equation*}
\int_0^\infty e^{\beta t}p_D(t,\,x,\,y_1)p_D(t,\,x,\,y_2)\,\d t\leq c_1[\frac{1}{\sqrt{\beta}}+\frac{1}{2\mu_1-\beta}].
\end{equation*}
\end{lemma}
\begin{proof}
As in the proof of Lemma \ref{upperbound-1}, we will split the integral as follows.
\begin{equation*}
\begin{aligned}
&\int_0^\infty e^{\beta t}p_D(t,\,x,\,y_1)p_D(t,\,x,\,y_2)\,\d t\\
&=\int_0^{t_0} e^{\beta t}p_D(t,\,x,\,y_1)p_D(t,\,x,\,y_2)\,\d t+\int_{t_0}^\infty e^{\beta t}p_D(t,\,x,\,y_1)p_D(t,\,x,\,y_2)\,\d t.
\end{aligned}
\end{equation*}
The appropriate bounds on the Dirichlet heat kernel and some computations prove the result.
\end{proof}
The following is similar to Lemma \ref{lowerbound}.
\begin{lemma}\label{lowerbound-1}
For all $\ve>0$, then there exists $t_0>0$ and $c_1>0$ depending on $\epsilon$ and $R$ such that for all $\beta>0$ and $x_1,\,x_2,\,y_1,\, y_2 \in \mathcal{B}(0, R-\ve)$, we have 
\begin{equation*}
\int_0^\infty e^{-\beta t} p_D(t,\,x_1,\,y_1) p_D(t,\,x_2,\,y_2)\, \d t \geq \frac{c_1 e^{-(\beta+2\mu_1)t_0}}{\beta+2\mu_1}.
\end{equation*}
\end{lemma}

\begin{proof}
If $x,\,y\in \mathcal{B}(0,R-\ve)$, then there exists $t_0>0$ and $c_2>0$ depending on $\epsilon$ and $R$ such that
\begin{equation*}
p_D(t,\,x,\,y)\geq c_2e^{-\mu_1t}\quad \text{for\,\,all}\quad t>t_0.
\end{equation*}
Thus, 
\begin{equation*}
\begin{aligned}
\int_0^\infty e^{-\beta t}&p_D(t,\,x_1,\,y_1) p_D(t,\,x_2,\,y_2)\, \d t\\&\geq \int_{t_0}^\infty e^{-\beta t}p_D(t,\,x_1,\,y_1) p_D(t,\,x_2,\,y_2)\, \d t\\
&\geq \frac{c_3e^{-(\beta+2\mu_1) t_0}}{\beta+2\mu_1}.
\end{aligned}
\end{equation*}
\end{proof}
Our final result of this section is a renewal-type inequality whose proof is straightforward and can be found in \cite{foonjose}.
\begin{proposition}\label{renew-lowerbound}
Suppose  that $f(t)$ is a non-negative integrable function satisfying
\begin{equation*}
f(t)\geq a+kb\int_0^t\frac{f(s)}{\sqrt{t-s}}\,\d s\quad \text{for\,all} \quad k,\, t>0,
\end{equation*}
where $a$ and $b$ are positive constants. Then, there exist positive constants $c_1$ and $c_2$ depending only on $a$ and $b$ such that,
\begin{equation*}
f(t)\geq c_1e^{c_2k^2t}\quad \text{for\,all}\quad t>0.
\end{equation*}
\end{proposition}

\section{Proofs of Theorems \ref{dirichlet}, \ref{dirichlet-drift} and Corollaries \ref{zero} and \ref{cor:dirichlet}.}

{\it Proof of Theorem \ref{dirichlet}.}
We start off with the mild formulation of the solution \eqref{mild:dirichlet} and take the second moment to end up with 
\begin{equation}\label{2-moment}
\begin{aligned}
\E|u_t(x)|^2&=|(\sG_Du)_t(x)|^2+\lambda^2\int_0^t\int_0^L p_{D}^2(t-s,\,x,\,y)\E|\sigma(u_s(y))|^2\,\d y\,\d s.\\
\end{aligned}
\end{equation}

Since it is known that $(\sG_Du)_t(x)$ decays exponentially fast with time, the above equation implies that $\E|u_t(x)|^2$ cannot decay faster than exponential.  This shows that 
\begin{equation*}
\limsup_{t\rightarrow \infty}\frac{1}{t}\log \E|u_t(x)|^2>-\infty.
\end{equation*}

Using the fact that $p_{D}(t,\,x,\,y)\leq p(t,\,x,\,y)$, the second moment of the solution satisfies 
\begin{equation*}
\begin{aligned}
\E|u_t(x)|^2&\leq |(\sG_Du)_t(x)|^2+\lambda^2\int_0^t\int_0^L p^2(t-s,\,x,\,y)\E|\sigma(u_s(y))|^2\,\d y\,\d s.\\
\end{aligned}
\end{equation*}
We can now follow the ideas of \cite{FK} to show that 
\begin{equation*}
\liminf_{t\rightarrow \infty}\frac{1}{t}\log \E|u_t(x)|^2<\infty.
\end{equation*}

We next show that 
\begin{equation}\label{sup}
\limsup_{t\rightarrow \infty}\frac{1}{t}\log \E|u_t(x)|^2<0.
\end{equation}
Choose $\beta \in  (0,2\mu_1)$, and let
\begin{equation*}
\|u\|_{2,\,\beta}:=\sup_{t>0}\sup_{x\in (0,\,L)}e^{\beta t}\E|u_t(x)|^2.
\end{equation*}
Since $\beta>0$, \eqref{sup} will be proved once we show that $\|u\|_{2,\,\beta}<\infty$.  Using this notation, together with the Lipschitz continuity of $\sigma$, we have 
\begin{equation*}
\begin{aligned}
\E|u_t(x)|^2&\leq |(\sG_Du)_t(x)|^2+\lambda^2 L_\sigma^2\int_0^t\int_0^L p_{D}^2(t-s,\,x,\,y)\E|u_s(y)|^2\,\d y\,\d s\\
&=I_1+I_2.
\end{aligned}
\end{equation*}
We bound the term $I_1$ first. $u_0(x)$ is bounded above, so we have
\begin{equation} \label{i1}
\begin{aligned}
I_1&\leq c_1e^{-\beta t}\left(\int_0^Le^{\beta t/2 }p_D(t,\,x,\,y)\,\d y\right)^2\\
&\leq c_2 e^{-\beta t},
\end{aligned}
\end{equation}
where we have use Lemma \ref{upperbound-2} for the last inequality. We now turn our attention to the second term. 
Using the fact that $\int_0^L p^2_D(t,x,y) dy=p_D(2t,x,x)$, we obtain 
\begin{equation*}
\begin{aligned}
I_2&= \lambda^2L_\sigma^2\int_0^t\int_0^L e^{\beta(t-s)}p_{D}^2(t-s,\,x,\,y)e^{-\beta(t-s)}\E|u_s(y)|^2\,\d y\,\d s\\
&\leq \|u\|_{2,\,\beta} \lambda^2L_\sigma^2 e^{-\beta t}\int_0^t\int_0^L e^{\beta s}p_{D}^2(s,\,x,\,y)\,\d y\,\d s\\
&\leq \|u\|_{2,\,\beta}\lambda^2L_\sigma^2e^{-\beta t}\int_0^\infty e^{\beta s}p_{D}(2s,\,x,\,x)\,\d s.\\
%&\leq \frac{c_3\|u\|_{2,\,\beta}\lambda^2L_\sigma^2e^{-\beta t}}{\sqrt{\beta}}.
\end{aligned}
\end{equation*}
We set $K_{\beta,\,\mu_1}:=[\frac{1}{\sqrt{\beta}}+\frac{1}{\mu_1-\beta}]$ and use Lemma \ref{upperbound-1} to obtain
\begin{equation*}
I_2\leq c_3K_{\beta,\,\mu_1}\|u\|_{2,\,\beta}\lambda^2L_\sigma^2e^{-\beta t}.
\end{equation*}
We now combine the above estimates and write
\begin{equation*}
\E|u_t(x)|^2\leq c_2 e^{-\beta t}+c_3K_{\beta,\,\mu_1}\|u\|_{2,\,\beta}\lambda^2L_\sigma^2e^{-\beta t},
\end{equation*}
which immediately yields
\begin{equation*}
\|u\|_{2,\,\beta}\leq c_2+c_3K_{\beta,\,\mu_1}\|u\|_{2,\,\beta}\lambda^2L_\sigma^2.
\end{equation*}
Therefore if we choose $\lambda$ small enough so that $c_3K_{\beta,\,\mu_1}\lambda^2L_\sigma^2<1$, then we have $\|u\|_{2,\,\beta}<\infty$. 

It now remains to show that 
\begin{equation}\label{inf}
\liminf_{t\rightarrow \infty}\frac{1}{t}\log \E|u_t(x)|^2>0.
\end{equation}
To ease the exposition, we introduce the following notation.  For any fixed $\beta>0$, set
\begin{equation*}
I_\beta:=\int_0^\infty e^{-\beta t}\inf_{x\in [\ve, L-\ve]}\E|u_t(x)|^2\,\d t.
\end{equation*}
Our starting point will again be \eqref{2-moment}, so upon using the lower bound on $\sigma$, we have 
\begin{equation}\label{lower}
\E|u_t(x)|^2\geq |(\sG_Du)_t(x)|^2+\lambda^2l_\sigma^2\int_0^t\int_0^Lp_{D}^2(t-s,\,x,\,y)\E|u_s(y)|^2\,\d y\,\d s.
\end{equation}
The second term on the right hand side is bounded below by 
\begin{equation*}
\begin{aligned}
 \lambda^2&l_\sigma^2\int_0^t\inf_{y\in [\ve,L-\ve]}\E|u_s(y)|^2\int_\ve^{L-\ve}p_{D}^2(t-s,\,x,\,y)\,\d y\,\d s\\
&\geq \lambda^2l_\sigma^2(L-2\ve)\int_0^t\inf_{y\in [\ve,L-\ve]}\E|u_s(y)|^2\inf_{x,\,y\in[\ve,L-\ve]}p_{D}^2(t-s,\,x,\,y)\,\d s,
\end{aligned}
\end{equation*}
for any $\ve>0$.  We now take infimum over $[\ve, L-\ve]$ on both sides of \eqref{lower}, multiply by $e^{-\beta t}$ and then integrate to obtain
\begin{equation*} 
\begin{aligned}
I_\beta &\geq \int_0^\infty e^{-\beta t}\inf_{x\in[\ve,L-\ve]}|(\sG_Du)_t(x)|^2 \,\d t
 \\
 & \qquad +\lambda^2 l_\sigma^2I_\beta (L-2\ve)\int_0^\infty e^{-\beta t}\inf_{x,\,y\in [\ve, L-\ve]}p^2_D(t,\,x,\,y)\, \d t\\
&=\tilde{I}_1+\tilde{I}_2.
\end{aligned}
\end{equation*}
We look at $\tilde{I}_2$ first. From Lemma \ref{lowerbound}, we have 
\begin{equation*}
\tilde{I}_2\geq \frac{c_4\lambda^2 l_\sigma^2(L-2\ve)I_\beta e^{-(\beta+2\mu_1)t_0}}{\beta+2\mu_1}.
\end{equation*}
Very similar computations to the proof of Lemma \ref{lowerbound} give us a lower bound on $\tilde{I}_1$. Indeed,  for $x\in[\ve,\,L-\ve]$, we have
\begin{equation*}
(\sG_Du)_t(x)\geq c_5\inf_{y\in[\ve,\,L-\ve]}p_D(t,\,x,\,y)\int_\ve^{L-\ve}u_0(y)\d y.
\end{equation*}
From our condition on $u_0$, we have 
\begin{equation*}
(\sG_Du)_t(x)\geq c_6e^{-\mu_1 t}\quad\text{for\,\,all}\quad t\geq t_0.
\end{equation*}
We thus have
\begin{equation} \label{aux2}
\tilde{I}_1\geq \frac{c_7e^{-(\beta+2\mu_1)t_0}}{\beta+2\mu_1}.
\end{equation}
Putting these estimates together, we obtain
\begin{equation*}
I_\beta\geq \frac{c_7e^{-(\beta+2\mu_1)t_0}}{\beta+2\mu_1}+\frac{c_4\lambda^2 l_\sigma^2(L-2\ve)I_\beta e^{-(\beta+2\mu_1)t_0}}{\beta+2\mu_1}.
\end{equation*}
We now choose $\lambda$ large enough so that $\frac{c_4\lambda^2 l_\sigma^2(L-2\ve)e^{-(\beta+2\mu_1)t_0}}{\beta+2\mu_1}>1$.  The above recursive inequality then shows that $I_\beta=\infty$. This proves \eqref{inf}, and concludes the proof of the theorem.
\qed
\vskip12pt
We will need the following proposition which concerns monotonicity of the moments with respect to $\lambda$.
\begin{proposition}\label{monotone}
Consider $u_t$, the unique solution to \eqref{eq:dirichlet}. Fix $x\in (0,\,L)$ and set $f_\lambda(t):=\E|u_t(x)|^p$ for any $p\geq 2$.  Then for all $t>0$, we  have
\begin{equation*}
f_{\lambda_1}(t)\geq f_{\lambda_2}(t)\quad\text{whenever}\quad \lambda_1\geq \lambda_2.
\end{equation*}
\end{proposition}
\begin{proof}
We follow the ideas of \cite{JKM} and discretise the equation as follows. For $n\geq 2$ and $k=1,2,..., n-1$, we set
\begin{equation*}
\d u_t\left(\frac{Lk}{n}\right)=\frac{1}{2}\Delta_nu_t\left(\frac{Lk}{n}\right) \d t+ \lambda \sqrt{N}
\sigma\left(u_t\left(\frac{Lk}{n} \right)\right)\,\d W_k(t), 
\end{equation*}
with $u_t\left(\frac{0}{n}\right)=0$ and $u_t\left(\frac{Ln}{n}\right)=0$ with $t>0$. The initial condition is given by $u_0\left(\frac{Lk}{n}\right)$. $W_k$ is an independent system of standard Brownian motions and 
\begin{equation*}
\Delta_nu_t\left(\frac{Lk}{n}\right):=n^2\left[u_t\left(\frac{L(k+1)}{n}\right)-2u_t\left(\frac{Lk}{n}\right)+u_t\left(\frac{L(k-1)}{n}\right)\right].
\end{equation*}
The above defines a discretised version of equation \eqref{eq:dirichlet} at points $\{\frac{Lk}{n}\}_{k=2}^{n}$.  For points $x\in (\frac{L(k-1)}{n}, \frac{Lk}{n})$, we consider a linear combinations of the discretised equation at the endpoints of the interval.  We denote the resulting process as $u^{(n)}_t(x)$. It is known that for all $t\geq 0$ and $x\in (0,\,L)$ this process converges to $u_t(x)$ in $L^p(\Omega)$ for all $p\geq2$; see for instance \cite{JKM} and references therein. On the other hand, by Theorem 1 of \cite{cox}, we have that the $p$th moment of $u_t^{n}(x)$ increases with $\lambda$. So by a limiting argument, the result follows. 
\end{proof}

{\it Proof of Corollary \ref{zero}.}
The fact that $\lambda_0\leq \lambda_1$ is obvious because otherwise, we would end up with a contradiction.
From Proposition \ref{monotone} we obtain the existence of two points in the interval $[\lambda_0, \lambda_1]$ such that the limits described in the corollary are zero. 
\qed
\vskip12pt
{\it Proof of Corollary \ref{cor:dirichlet}.} It suffices to use the fact that
$$
(L-2\epsilon)\inf_{x \in (\epsilon, L-\epsilon) }\E \vert u_t(x) \vert^2 \leq \int_0^L \E \vert u_t(x) \vert^2 \,\d x\leq L \sup_{x \in (0,L)}\E \vert u_t(x) \vert^2 ,
$$
together with Theorem \ref{dirichlet} and the definition of $\mathcal{E}_t(\lambda)$.
\qed
\vskip12pt
{\it Proof of Theorem \ref{dirichlet-drift}.}
The proof is very similar to that of Theorem \ref{dirichlet}, so we just indicate the difference.  As before, we use Walsh's approach to define the mild solution as the solution to the following integral equation,
\begin{equation*}
u_t(x)=
(\sG^*_Du)_t(x)+ \lambda \int_0^L\int_0^t p^*_D(t-s,\,x,\,y)\sigma(u_s(y))W(\d s,\,\d y),
\end{equation*}
where $p^*_D(t,x,y)=e^{\mu t}p_D(t,x,y)$, and
\begin{equation*}
(\sG^*_D u)_t(x):=\int_0^L u_0(y)p^*_D(t,\,x,\,y)\,\d y.
\end{equation*}
Let $\mu\in \R$ be fixed. Using the methods as in the proofs of Lemmas \ref{upperbound-1} and \ref{upperbound-2}, we have that for all $\beta>0$ such that $0<\beta+\mu<\mu_1$,
\begin{equation*}
\int_0^\infty e^{\beta t}p^*_D(t,x,x)\, \d t\leq c_1[\frac{1}{\sqrt{\beta+\mu}}+\frac{1}{\mu_1-(\beta+\mu)}],
\end{equation*}
and 
\begin{equation*}
\sup_{t>0}\int_0^L e^{\beta t} p^*_D(t,\,x,\,y)\,d y\leq c_2,
\end{equation*}
where $c_1$ and $c_2$ are positive constants depending on $L$ and $\mu$.  The proof of upper bound now follows exactly as that of Theorem \ref{dirichlet}.  For the lower bound, we follow the proof of Lemma \ref{lowerbound} to obtain that if $\beta+2\mu_1>2\mu$, then for $x,\,y \in[\ve,\,L-\ve]$, we have 
\begin{equation*}
\int_0^\infty e^{-\beta t} (p^*_D(t,\,x,\,y))^2\, \d t \geq \frac{c_1 e^{-(\beta-2\mu+2\mu_1)t_0}}{\beta-2\mu+2\mu_1}.
\end{equation*}
Now the rest of the proof follows from the second part of Theorem \ref{dirichlet}.  For the proof of the last statement, it suffices to look at the dependence on $\mu$ of $\lambda_0$ and $\lambda_1$.
\qed

\section{Proof of Theorem \ref{neumann} and Corollary \ref{cor:neumann}.}

{\it Proof of Theorem \ref{neumann}.}
We start with the mild formulation of the solution, take second moment and use the growth conditions on $\sigma$ to obtain 
\begin{equation*}
\begin{aligned}
\E|u_t(x)|^2&=|(\sG_Nu)_t(x)|^2+\lambda^2\int_0^t\int_0^L p_{N}^2(t-s,\,x,\,y)\E|\sigma(u_s(y))|^2\,\d y\,\d s\\
&\geq |(\sG_Nu)_t(x)|^2+\lambda^2l_\sigma^2\int_0^t\int_0^L p_{N}^2(t-s,\,x,\,y)\E|u_s(y)|^2\,\d y\,\d s\\
&=I_1+I_2.
\end{aligned}
\end{equation*}
We find a lower bound on $I_1$ first. We now use Lemma \ref{neu:lower} and the condition on $u_0$ to obtain the following
\begin{equation*}
I_1=\left|\int_0^L p_N(t,\,x,\,y)u_0(y)\,\d y \right|^2\geq c_1 \left|\int_0^Lu_0(y)\,\d y \right|^2\geq c_2.
\end{equation*}
We now look at the second term $I_2$.
\begin{equation*}
\begin{aligned}
I_2&\geq \lambda^2 l_\sigma^2\int_0^t\inf_{y\in[0,\,L]}\E|u_s(y)|^2 p_N(t-s,\,y,\,y)\, \d s\\
&\geq c_3\lambda^2 l_\sigma^2\int_0^t\frac{1}{\sqrt{t-s}}\inf_{y\in[0,\,L]}\E|u_s(y)|^2\, \d s.
\end{aligned}
\end{equation*}
Combining the above estimates, we obtain 
\begin{equation*}
\inf_{x\in [0,\,L]}\E|u_t(x)|^2\geq c_3+c_4\lambda^2 l_\sigma^2\int_0^t\frac{1}{\sqrt{t-s}}\inf_{y\in [0,\,L]}\E|u_s(y)|^2\, \d s.
\end{equation*}
We now use Proposition \ref{renew-lowerbound} to conclude the first part of the result.
Similar ideas to those used in \cite{Khoshnevisan:2013ab} can be used to show that 
\begin{equation*}
\liminf_{t\rightarrow \infty}\frac{1}{t}\log \E|u_t(x)|^2< \infty \quad\text{for\,all}\quad \lambda>0.
\end{equation*}
We leave it to the reader to fill in the details.
\qed
\vskip12pt
{\it Proof of Corollary \ref{cor:neumann}.}
The proof is quite straightforward, since we have 
\begin{equation*}
\begin{aligned}
\int_0^L\E|u_t(x)|^2\,\d x\geq L\inf_{x\in[0,\,L]}\E|u_t(x)|^2.
\end{aligned}
\end{equation*}
The result now follows from Theorem \ref{neumann}.
\qed

\section{Proofs of Theorem \ref{coloured} and Corollary \ref{cor:coloured}}

{\it Proof of Theorem \ref{coloured}.}
The starting point for the proof is the mild formulation given by \eqref{mild:coloured} which after taking the second moment gives 
\begin{equation*}
\begin{aligned}
&\E|u_t(x)|^2=|(\sG_Du)_t(x)|^2 +\lambda^2\int_0^t\int_{\mathcal{B}(0,R)\times\mathcal{B}(0,R) } p_{D}(t-s,\,x,\,y_1)p_{D}(t-s,\,x,\,y_2) \\
&\qquad \qquad \times f(y_1-y_2) \E|\sigma(u_s(y_1))\sigma(u_s(y_2))|\,\d y_1\,\d y_2
\,\d s.
\end{aligned}
\end{equation*}
Using the Lipschitz continuity of $\sigma$, we get that
\begin{equation*} 
\begin{aligned}
\E|u_t(x)|^2&\leq |(\sG_Du)_t(x)|^2 +\lambda^2 L_{\sigma}^2
\int_0^t\int_{\mathcal{B}(0,R) \times \mathcal{B}(0,R) } p_{D}(t-s,\,x,\,y_1)\\
& \qquad \qquad \times p_{D}(t-s,\,x,\,y_2)f(y_1-y_2) \E|u_s(y_1)u_s(y_2)|\,\d y_1\,\d y_2
\,\d s \\
& =I_1+ I_2.
\end{aligned}
\end{equation*}
We now use the fact that $u_0$ is bounded together with Lemma \ref{upperbound-4} to obtain 
$$
I_1 \leq c_1 e^{-\beta t},
$$
whenever $\beta \in (0, 2 \mu_1)$. We have used a similar argument which led to \eqref{i1}.  In order to deal with $I_2$, we set 
\begin{equation*}
\|u\|_{2,\,\beta}:=\sup_{t>0}\sup_{x\in \mathcal{B}(0,R)}e^{\beta t}\E|u_t(x)|^2.
\end{equation*}
and observe that
\begin{equation*}
\begin{aligned}
&I_2\leq\lambda^2 L_{\sigma}^2
\int_0^t\int_{\mathcal{B}(0,R) \times \mathcal{B}(0,R) } p_{D}(t-s,\,x,\,y_1) p_{D}(t-s,\,x,\,y_2)f(y_1-y_2) \,\d y_1\,\d y_2\\
& \qquad \qquad \times e^{\beta(t-s)}e^{-\beta(t-s)} \sup_{y \in \mathcal{B}(0,R)} \E\vert u_s(y) \vert^2
\,\d s \\
&\leq \|u\|_{2,\,\beta} \lambda^2 L_{\sigma}^2e^{-\beta t}
\int_0^t  \int_{\mathcal{B}(0,R) \times \mathcal{B}(0,R) }   p_{D}(s,\,x,\,y_1) p_{D}(s,\,x,\,y_2)f(y_1-y_2) \,\d y_1\,\d y_2\\
&\qquad \qquad \times  e^{\beta s} \,\d s.
\end{aligned}
\end{equation*}
Next, we see that 
\begin{equation*}
\begin{aligned}
\int_{\mathcal{B}(0,R) \times \mathcal{B}(0,R) }&p_{D}(s,\,x,\,y_1) p_{D}(s,\,x,\,y_2)f(y_1-y_2) \,\d y_1\,\d y_2\\
&\leq \sup_{y_1,\,y_2\in \mathcal{B}(0,R)}p_{D}(s,\,x,\,y_1) p_{D}(s,\,x,\,y_2) \\
&\times \int_{\mathcal{B}(0,R) \times \mathcal{B}(0,R) }f(y_1-y_2)\,\d y_1\,\d y_2.
\end{aligned}
\end{equation*}
We now use the above, Lemma \ref{upperbound-3} and the fact that $$\int_{\mathcal{B}(0,R) \times \mathcal{B}(0,R)} f(y_1-y_2) dy_1 \, dy_2 < \infty,$$ to conclude that
\begin{equation*}
I_2 \leq c_2\tilde{K}_{\beta,\,\mu_1}\|u\|_{2,\,\beta}\lambda^2L_\sigma^2e^{-\beta t},
\end{equation*}
where $\tilde{K}_{\beta,\,\mu_1}:=[\frac{1}{\sqrt{\beta}}+\frac{1}{2\mu_1-\beta}]$.  Combining the above estimates, we obtain
\begin{equation*}
\E|u_t(x)|^2\leq c_1 e^{-\beta t}+c_2\tilde{K}_{\beta,\,\mu_1}\|u\|_{2,\,\beta}\lambda^2L_\sigma^2e^{-\beta t},
\end{equation*}
which yields
\begin{equation*}
\|u\|_{2,\,\beta}\leq c_1+c_2\tilde{K}_{\beta,\,\mu_1}\|u\|_{2,\,\beta}\lambda^2L_\sigma^2.
\end{equation*}
Therefore if $\lambda$ is chosen such that $c_2 \tilde{K}_{\beta,\,\mu_1}\lambda^2L_\sigma^2<1$, we will have $\|u\|_{2,\,\beta}<\infty$, which shows that 
\begin{equation*}
\limsup_{t\rightarrow \infty}\frac{1}{t}\log \E|u_t(x)|^2<0.
\end{equation*}
The proof of 
\begin{equation*}
\limsup_{t\rightarrow \infty}\frac{1}{t}\log \E|u_t(x)|^2>-\infty
\end{equation*}
follows from the fact the first term $|(\sG_Du)_t(x)|^2$ decays exponentially fast and that $\E|u_t(x)|^2\geq |(\sG_Du)_t(x)|^2$.  We next show that for $\lambda$ large enough, 
\begin{equation*}
\liminf_{t\rightarrow \infty}\frac{1}{t}\log \E|u_t(x)|^2>0.
\end{equation*}
Instead of  looking at the second moment directly, we will look at an auxiliary quantity $J_\beta$ defined as follows.
Set
\begin{equation*} 
J_{\beta}:=\int_0^{\infty} e^{-\beta s}
\inf_{y_1,y_2 \in  \mathcal{B}(0,R-\ve)}\E|u_s(y_1)u_s(y_2)|  \d s.
\end{equation*}
We claim that $J_{\beta}=\infty$ for some $\beta>0$.  To prove this claim, we start with
\begin{equation} \label{aux3}
\begin{aligned}
\E(u_t(x_1)&u_t(x_2))=(\sG_Du)_t(x_1)(\sG_Du)_t(x_2)+\lambda^2\int_0^t\int_{\mathcal{B}(0,R)^2} p_{D}(t-s,\,x_1,\,y_1)\\
&\times p_{D}(t-s,\,x_2,\,y_2) f(y_1-y_2) \E\left(\sigma(u_s(y_1))\sigma(u_s(y_2))\right)\d y_1\,\d y_2 \,\d s. 
\end{aligned}
\end{equation}
Using the lower bound for $\sigma$, we see that the second term of the above display is bounded below by 
\begin{equation*}
\begin{aligned}
\lambda^2 l_{\sigma}^2 &V^2_{R-\epsilon} \int_0^t  \inf_{y_1,y_2 \in  \mathcal{B}(0,R-\ve)}  \E|u_s(y_1)u_s(y_2)|  \\
 & \times \inf_{y_1,y_2 \in  \mathcal{B}(0,R-\ve)}  p_{D}(t-s,\,x_1,\,y_1) p_{D}(t-s,\,x_2,\,y_2) \,\d s \\
 &\times \inf_{y_1,y_2 \in  \mathcal{B}(0,R)} f(y_1-y_2),
\end{aligned}
\end{equation*}
where $V_{R-\ve}$ denotes the volume of the ball $\mathcal{B}(0,R-\ve)$.  We now take infimum over $\mathcal{B}(0,R-\ve)$ on both sides of \eqref{aux3}, multiply by $e^{-\beta t}$ and then integrate to obtain
\begin{equation*} 
\begin{aligned}
J_\beta &\geq \int_0^{\infty} e^{-\beta s}
\inf_{x_1,x_2 \in  \mathcal{B}(0,R-\ve)} (\sG_Du)_t(x_1)(\sG_Du)_t(x_2) \d s 
 \\
 & \qquad +\lambda^2 l_\sigma^2 V^2_{R-\ve} J_{\beta} \int_0^{\infty} e^{-\beta t} \inf_{x, y_1,y_2 \in  \mathcal{B}(0,R-\ve)}p_{D}(t,\,x,\,y_1)
p_{D}(t,\,x,\,y_2)\d t,
\end{aligned}
\end{equation*}
where we have used the fact that
\begin{equation} \label{f}
 \inf_{y_1,y_2 \in  \mathcal{B}(0,R)} f(y_1-y_2)>0.
 \end{equation}
 
As before, we apply Lemma \ref{lowerbound-1} to get that
\begin{equation*} 
J_\beta \geq \frac{c_5 e^{-(\beta+2\mu_1)t_0}}{\beta+2\mu_1}+\lambda^2 l_\sigma^2 V^2_{R-\ve} J_{\beta} \frac{c_6 e^{-(\beta+2\mu_1)t_0}}{\beta+2\mu_1}.
\end{equation*}
We now choose $\lambda$ so that $\frac{c_6 \lambda^2 l_\sigma^2 V^2_{R-\ve} e^{-(\beta+2\mu_1)t_0}}{\beta+2\mu_1}>1$ which implies that $J_{\beta}=\infty$.  We now show that this fact implies our result. 
We have 
\begin{equation} \label{aux1}
\begin{aligned}
\E|u_t(x)|^2&\geq |(\sG_Du)_t(x)|^2 +\lambda^2 l_{\sigma}^2
\int_0^t\int_{\mathcal{B}(0,R) \times \mathcal{B}(0,R) } p_{D}(t-s,\,x,\,y_1)\\
& \qquad  \times p_{D}(t-s,\,x,\,y_2)f(y_1-y_2) \E|u_s(y_1)u_s(y_2)|\,\d y_1\,\d y_2
\,\d s \\
& =\tilde{I}_1+ \tilde{I}_2.
\end{aligned}
\end{equation}
Fix $\ve>0$.  Note that by (\ref{f}), we have 
\begin{equation*}
\begin{aligned}
\tilde{I}_2 &\geq c_2\lambda^2 l_\sigma^2 \int_0^t\inf_{y_1,y_2 \in  \mathcal{B}(0,R-\ve)}\E|u_s(y_1)u_s(y_2)|
\int_{\mathcal{B}(0,R-\ve) \times \mathcal{B}(0,R-\ve) } p_{D}(t-s,\,x,\,y_1)\\
& \qquad \qquad \times p_{D}(t-s,\,x,\,y_2) \,\d y_1\,\d y_2
\,\d s \\
&\geq c_2\lambda^2 l_\sigma^2 V^2_{R-\ve} \int_0^t
\inf_{y_1,y_2 \in  \mathcal{B}(0,R-\ve)}\E|u_s(y_1)u_s(y_2)| \\
&\qquad \qquad \times \inf_{y_1,y_2 \in  \mathcal{B}(0,R-\ve)}p_{D}(t-s,\,x,\,y_1)
p_{D}(t-s,\,x,\,y_2)\d s.\\
\end{aligned}
\end{equation*}
  We now take infimum over $\mathcal{B}(0,R-\ve)$ on both sides of \eqref{aux1}, multiply by $e^{-\beta t}$ and then integrate to obtain
\begin{equation*} 
\begin{aligned}
I_\beta &\geq \int_0^\infty e^{-\beta t}\inf_{x\in\mathcal{B}(0,R-\ve)}|(\sG_Du)_t(x)|^2 \,\d t
 \\
 & \qquad +c_2\lambda^2 l_\sigma^2 V^2_{R-\ve} \int_0^{\infty} e^{-\beta s}
\inf_{y_1,y_2 \in  \mathcal{B}(0,R-\ve)}\E|u_s(y_1)u_s(y_2)| \d s\\
&\qquad \qquad \times \int_0^{\infty} e^{-\beta t} \inf_{x, y_1,y_2 \in  \mathcal{B}(0,R-\ve)}p_{D}(t,\,x,\,y_1)
p_{D}(t,\,x,\,y_2) \d t,
\end{aligned}
\end{equation*}
where we have used the notation 
\begin{equation*}
I_\beta:=\int_0^\infty e^{-\beta t}\inf_{x\in \mathcal{B}(0,R-\ve)}\E|u_t(x)|^2\,\d t.
\end{equation*}

We now appeal to Lemma \ref{lowerbound-1} and use similar computations to the ones that lead to (\ref{aux2}), to get that
\begin{equation} \label{aux4}
\begin{aligned}
I_\beta &\geq  \frac{c_3 e^{-(\beta+2\mu_1)t_0}}{\beta+2\mu_1}+\lambda^2 l_\sigma^2 V_{R-\ve} \int_0^{\infty} e^{-\beta s}
\inf_{y_1,y_2 \in  \mathcal{B}(0,R-\ve)}\E|u_s(y_1)u_s(y_2)| \d s\\
&\qquad \qquad \times  \frac{c_4 e^{-(\beta+2\mu_1)t_0}}{\beta+2\mu_1}.
\end{aligned}
\end{equation}
Using the definition of $J_{\beta}$, we have 
\begin{equation} 
\begin{aligned}
I_\beta &\geq  \frac{c_3 e^{-(\beta+2\mu_1)t_0}}{\beta+2\mu_1}+\lambda^2 l_\sigma^2 V_{R-\ve} J_{\beta}\frac{c_4 e^{-(\beta+2\mu_1)t_0}}{\beta+2\mu_1}.
\end{aligned}
\end{equation}
This proves the result since $J_\beta=\infty$ for $\lambda$ large enough.  The proof of the theorem will be complete once we have that
\begin{equation*}
\liminf_{t\rightarrow \infty}\frac{1}{t}\log \E|u_t(x)|^2<\infty.
\end{equation*}
But this follows from ideas from \cite{FK2}. We leave it to the reader to fill in the details.
\qed
\vskip12pt
{\it Proof of Corollary \ref{cor:coloured}.} It suffices to use the fact that
$$
V_{R-\epsilon}\inf_{x \in \mathcal{B}(0, R-\epsilon)}\E \vert u_t(x) \vert^2 \leq \int_{\mathcal{B}(0,R)} \E \vert u_t(x) \vert^2 \d x \leq 
V_R \sup_{x \in \mathcal{B}(0,R)}\E \vert u_t(x) \vert^2 ,
$$
together with Theorem \ref{coloured} and the definition of $\mathcal{E}_t(\lambda)$.
\qed

\section{An extension}
All of the main results described in the introduction involve the second moment. In this section, we show how one can get analogous results for higher moments. We will focus only on the solution to \eqref{dirichlet}. Analogous results for the other equations will follow from similar ideas. 
\begin{theorem}\label{dirichlet_higher}
Let $u_t$ be the unique solution to \eqref{dirichlet}. Then for all $p \geq 2$, there exists $\lambda_0(p)>0$ such that for all $\lambda<\lambda_0(p)$ and $x \in (0,L)$, 
\begin{equation*}
-\infty<\limsup_{t\rightarrow \infty}\frac{1}{t}\log \E|u_t(x)|^p<0.
\end{equation*}
On the other hand, for all $\ve>0$, there exists $\lambda_1(p)>0$ such that for all $\lambda>\lambda_1(p)$, we have 
\begin{equation*}
0<\liminf_{t\rightarrow \infty}\frac{1}{t}\log \E|u_t(x)|^p<\infty,
\end{equation*}
whenever $x\in [\ve,L-\ve]$.  Moreover, $\lambda_0(p)$ decreases with $p$.
\end{theorem}

\begin{proof}
Choose $\beta \in  (0,2\mu_1)$, and let
\begin{equation*}
\|u\|_{p,\,\beta}:=\sup_{t>0}\sup_{x\in \R}e^{p\beta t/2}\E|u_t(x)|^p.
\end{equation*}
Since $\beta>0$, the first part of the theorem will be proved once we show that $\|u\|_{p,\,\beta}<\infty$.  Using this notation, Burkh\"older's inequality together with the condition on $\sigma$, we have 
\begin{equation*}
\begin{aligned}
\E|u_t(x)|^p&\leq c_p\left\{|(\sG_Du)_t(x)|^p+\left(\lambda^2 L_\sigma^2\int_0^t\int_0^L p_{D}^2(t-s,\,x,\,y)\E|u_s(y)|^2\,\d y\,\d s\right)^{p/2}\right\}\\
&=I_1+I_2.
\end{aligned}
\end{equation*}
We bound the term $I_1$ first. Using the fact $u_0(x)$ is bounded, we have
\begin{equation*} 
\begin{aligned}
I_1&\leq c_1e^{-p\beta t/2}\left(\int_0^Le^{\beta t/2 }p_D(t,\,x,\,y)\,\d y\right)^p\\
&\leq c_2 e^{-p\beta t/2},
\end{aligned}
\end{equation*}
where we have use Lemma \ref{upperbound-2}  for the last inequality. We now turn our attention to the second term. 
Using the fact that $\int_0^L p^2_D(t,x,y) dy=p_D(2t,x,x)$, we get that
\begin{equation*}
\begin{aligned}
I_2&= \left(\lambda^2L_\sigma^2\int_0^t\int_0^L e^{\beta(t-s)}p_{D}^2(t-s,\,x,\,y)e^{-\beta(t-s)}\E|u_s(y)|^2\,\d y\,\d s\right)^{p/2}\\
&\leq \|u\|_{p,\,\beta} \left(\lambda^2L_\sigma^2 e^{-\beta t}\int_0^t\int_0^L e^{\beta s}p_{D}^2(s,\,x,\,y)\,\d y\,\d s\right)^{p/2}\\
&\leq \|u\|_{p,\,\beta} \left(\lambda^2L_\sigma^2e^{-\beta t}\int_0^\infty e^{\beta s}p_{D}(2s,\,x,\,x)\,\d s\right)^{p/2}\\
&\leq c_3K^{p/2}_{\beta, \mu_1}\|u\|_{p,\,\beta} \lambda^pL_\sigma^pe^{-p\beta t/2},
\end{aligned}
\end{equation*}
where we have use Lemma \ref{upperbound-1} for the last inequality.  Combining the above estimates, we obtain
\begin{equation*}
\E|u_t(x)|^p\leq c_2 e^{-p\beta t/2}+ c_3K^{p/2}_{\beta, \mu_1}\|u\|_{p,\,\beta} \lambda^pL_\sigma^pe^{-p\beta t/2},
\end{equation*}
which immediately yields
\begin{equation*}
\|u\|_{p,\,\beta}\leq c_2+c_3K^{p/2}_{\beta, \mu_1}\|u\|_{p,\,\beta}\lambda^pL_\sigma^p.
\end{equation*}
Therefore if $c_3K^{p/2}_{\beta, \mu_1}\lambda^pL_\sigma^p<1$, then we have $\|u\|_{p,\,\beta}<\infty$. By keeping track of the dependence on $p$, we see that $c_{3}$ increases with $p$. Therefore, $\lambda$ decreases with $p$. The proof of 
\begin{equation*}
\limsup_{t\rightarrow \infty}\frac{1}{t}\log \E|u_t(x)|^p>-\infty
\end{equation*}
is the same as in the special case of $p=2$.  We omit it here.  By Cauchy-Schwarz inequality, we have
$$
\E|u_t(x)|^p \geq (\E|u_t(x)|^2)^{p/2}.
$$
This implies that $\E|u_t(x)|^p$ grows exponentially whenever $\E|u_t(x)|^2$ does. The proof of
\begin{equation*}
\liminf_{t\rightarrow \infty}\frac{1}{t}\log \E|u_t(x)|^p<\infty,
\end{equation*}
follows from the ideas of \cite{FK}. 
\end{proof}

The above theorem deals with \eqref{eq:dirichlet} only.  But it is clear from the proof, one can also obtain similar results for \eqref{eq:neumann} and \eqref{eq:dir}. We leave this to the reader.

\section{Appendix}

As mentioned in the introduction, we settle the issue of existence-uniqueness of \eqref{eq:dir} here.
We begin with the following lemma. Recall that $p_D(t,\,x,\,y)$ denotes the Dirichlet heat kernel.

\begin{lemma} \label{lem:preli_ex}
Suppose that for any $\epsilon>0$, we have 
\begin{equation} \label{a1}
\int_{\vert x \vert \leq \epsilon} f(x) \log \frac{1}{\vert x \vert} dx < \infty, \quad \text{ if } d=2,
\end{equation}
or
\begin{equation} \label{a2}
\int_{\vert x \vert \leq \epsilon} f(x) \frac{1}{\vert x \vert^{d-2}} dx < \infty, \quad \text{ if } d\geq 3,
\end{equation}
for all $t>0$ and $x \in \mathcal{B}(0,\,R)$,
\begin{equation*}
\int_0^t \int_{\mathcal{B}(0,\,R)\times \mathcal{B}(0,\,R)} p_D(s,\,x,\,y)p_D(s,\,x,\,z)f(y,\,z)\,\d y\,\d z\,\d s\leq c,
\end{equation*}
where $c$ is some positive constant.
\end{lemma}

\begin{proof}
We begin by noting that $p_D(t,\,x,\,y)\leq p(t,\,x,\,y)$ from which we have 
\begin{equation*}
\begin{aligned}
\int_0^t\int_{\mathcal{B}(0,\,R)\times \mathcal{B}(0,\,R)}&p_D(s,\,x,\,y)p_D(s,\,x,\,z)f(y,\,z)\,\d y\,\d z\,\d s\\
&\leq
\int_0^\infty\int_{\R^d\times \R^d}p(s,\,x,\,y)p(s,\,x,\,z)f(y,\,z)\,\d y\,\d z\,\d s\\
&\leq c_1\int_{\R^d}\frac{\hat{f}(\xi)}{1+|\xi|^2}\,\d \xi,
\end{aligned}
\end{equation*}
where $\hat{f}$ is the Fourier transform of $f$.
But it is known that the right hand side of the above is bounded if and only if the function $f$ satisfies assumption \eqref{a1} and \eqref{a2}; see \cite{Dalang} for details.
\end{proof}
For $d=1$, the quantity 
\begin{equation*}
\int_0^t \int_{\mathcal{B}(0,\,R)\times \mathcal{B}(0,\,R)} p_D(s,\,x,\,y)p_D(s,\,x,\,z)f(y,\,z)\,\d y\,\d z\,\d s
\end{equation*}
is always bounded. 
We can now state and prove the main existence and uniqueness theorem,
which follows from a standard Picard iteration scheme.
\begin{theorem}
Suppose that $f$ satisfies condition \eqref{a1} when $d=2$ and \eqref{a2} when $d\geq 3$ and all the other conditions are as in the introduction of the paper, then \eqref{eq:dir} has a unique random field solution satisfying 
\begin{equation*}
\sup_{0\leq t\leq T,\,x\in \mathcal{B}(0,\,R)}\E|u_t(x)|^p<\infty,
\end{equation*}
for any $p\geq 2$ and $T < \infty$. No integrability condition is needed when $d=1$.
\end{theorem}
\begin{proof}
Set $u_t^{(0)}(x):=(\sG_Du)_t(x)$ and for $n\geq 1$, set
\begin{equation*}
u^{(n)}_t(x)=
(\sG_Du)_t(x)+ \lambda\int_0^t \int_{\mathcal{B}(0,R)} p_D(t-s,\,x,\,y)\sigma(u^{(n-1)}_s(y))F(\d s,\,\d y).
\end{equation*}
We have by Burkh\"older's inequality, 
\begin{equation*}
\begin{split}
&\E|u^{(n)}_t(x)-u^{(n-1)}_t(x)|^p\leq c_p\lambda^p\big[ \int_0^t\int_{\mathcal{B}(0,R)\times\mathcal{B}(0,R)} p_D(t-s,\,x,\,y)p_D(t-s,\,x,\,z)\\
&\times \E|\sigma(u^{(n-1)}_s(y))-\sigma(u^{(n-2)}_s(y))||\sigma(u^{(n-1)}_s(z))-\sigma(u^{(n-2)}_s(z))|f(y,\,z)\,\d y\,\d z\,\d s \big]^{p/2},
\end{split}
\end{equation*}
where $c_p$ is some positive constant. Using 
\begin{equation*}
d_n(t,\,x):=u^{(n)}_t(x)-u^{(n-1)}_t(x),
\end{equation*}
and the assumption on $\sigma$, we obtain
\begin{equation*}
\begin{split}
&\E|d_n(t,\,x)|^p\leq c_p\lambda^pL^p_\sigma\big[ \int_0^t\int_{\mathcal{B}(0,R)\times\mathcal{B}(0,R)} p_D(t-s,\,x,\,y)p_D(t-s,\,x,\,z)\\
&\times \sup_{x\in \mathcal{B}(0,R)}\E|d_{n-1}(s,\,x)|^2f(y,\,z)\,\d y\,\d z\,\d s \big]^{p/2}
\end{split}
\end{equation*}
We set
\begin{equation*}
F(t):=\int_0^t\int_{\mathcal{B}(0,R)\times\mathcal{B}(0,R)} p_D(t-s,\,x,\,y)p_D(t-s,\,x,\,z)f(y,\,z)\,\d y\,\d z\,\d s.
\end{equation*}
We now use H\"older's inequality to obtain
\begin{equation*}
\sup_{x\in \mathcal{B}(0,R)}\E|d_n(t,\,x)|^p\leq c_p\lambda^pL^p_\sigma F(t)^{p/2-1}\int_0^t \sup_{x\in \mathcal{B}(0,R)} \E|d_{n-1}(s,\,x)|^p\d s
\end{equation*}
By Lemma \ref{lem:preli_ex}, $F(t)$ is bounded.  We thus obtain  
\begin{equation*}
\sup_{x\in \mathcal{B}(0,R)}\E|d_n(t,\,x)|^p\leq C\int_0^t \sup_{x\in \mathcal{B}(0,R)} \E|d_{n-1}(s,\,x)|^p\d s, 
\end{equation*}
for some constant $C>0$. Now Gronwall's lemma gives convergence of the sequence $u^{(n)}_t(x)$ in the $p$th moment uniformly over $[0,T] \times \mathcal{B}(0,R)$, thus showing existence of a solution. Uniqueness follows from a well known argument. See \cite{minicourse} for details.

\end{proof}

\bibliography{Foon}
\end{document}